\documentclass{amsart} 
\usepackage{amssymb}
\usepackage{amsthm} 
\usepackage[usenames,dvipsnames]{color}
 \usepackage[square, numbers]{natbib}
 \newlength\mylen
\setcitestyle{citesep={,\kern-\mylen}}
\usepackage{xurl}
\usepackage[colorlinks,citecolor=blue,linkcolor=blue,urlcolor=black]{hyperref}
\usepackage[capitalise]{cleveref}
\usepackage{thmtools}
\usepackage{orcidlink}

\author[N.~Greenberg]{Noam Greenberg \orcidlink{0000-0003-2917-3848}} 
\address{School of Mathematics and Statistics, Victoria University of Wellington, Wellington, New Zealand}
\email{noam.greenberg@vuw.ac.nz}

\author[A.~Nies]{Andr\'e Nies \orcidlink{0000-0002-0666-5180}} 
\address{School of Computer Science, University of Auckland, Auckland, New Zealand}
\email{andre@cs.auckland.ac.nz}
\thanks{The authors gratefully acknowledge the support of the Marsden fund of New Zealand. Some of this work was carried out at the Erwin Schr\"odinger Institute (ESI), Vienna  as part of the  ESI program on Reverse Mathematics in  2025.}

\author[D.~Turetsky]{Daniel Turetsky \orcidlink{0000-0002-4154-4800}} 
\address{School of Mathematics and Statistics, Victoria University of Wellington, Wellington, New Zealand}
\email{dan.turetsky@vuw.ac.nz}

\title{Characterising SJT reducibility}

\newcommand{\seq}[1]{{\left\langle{#1}\right\rangle}}
\newcommand{\uhr}[1]{\! \upharpoonright_{#1}}

\newcommand \s{\sigma}

\newcommand{\converge}{\!\!\downarrow}

\renewcommand \phi {\varphi}

\newcommand{\DII}{\Delta^0_2}
\newcommand{\NN}{{\mathbb{N}}}

\newcommand{\QQ}{{\mathbb{Q}}}

\newcommand{\sub}{\subseteq}
\newcommand{\sN}[1]{_{#1\in \NN}}
 
\newcommand{\cantor}{2^{ \omega}}

\newcommand{\ML}{Martin-L\"of}

\newcommand{\SI}[1]{\Sigma^0_{#1}}
\newcommand{\PI}[1]{\Pi^0_{#1}}

\newcommand{\bi}{\begin{itemize}}
\newcommand{\ei}{\end{itemize}}

\newcommand{\Halt}{{\ES'}}
\newcommand{\ES}{\emptyset}

\newcommand{\fa}{\forall}
\newcommand{\lep}{\le^+}

\newcommand{\la}{\langle}
\newcommand{\ra}{\rangle}
\newcommand{\Kuc}{Ku{\v c}era}

\newcommand{\leT}{\le_{\mathrm{T}}}

\newcommand{\ltt}{\le_{\mathrm{tt}}}

\newcommand{\MLR}{\mbox{\rm \textsf{MLR}}}
\newcommand{\Om}{\Omega}

\newcommand{\leb}{\mathbf{\lambda}}
\newcommand{\lwtt}{\le_{\mathrm{wtt}}}

\newcommand\+[1]{\mathcal{#1}}

\newcommand{\LR}{\Leftrightarrow}

\newcommand{\RA}{\Rightarrow}

\DeclareMathOperator \BLR{BLR}

\newcommand \DemBLR{\textup{Demuth}_{\textup{BLR}}}

 \newcommand{\bc}{\begin{center}}
\newcommand{\ec}{\end{center}}

\newcommand{\SJR}{\le_{\textup{\scriptsize{SJT}}}}
\newcommand{\LRR}{{\textup{\scriptsize{LR}}}}
 
\newcommand{\cost}{\mathbf{c}}

\renewcommand{\epsilon}{\varepsilon}
\newcommand{\rest}{\restriction}

\newcommand{\Tur}{{\textup{T}}}
\newcommand{\floor}[1]{\left\lfloor{#1}\right\rfloor}

\newcommand{\nge}{\not\ge}

\theoremstyle{plain}
\newtheorem{theorem}{Theorem}[section]

\newcounter{claimCounter}[theorem]

\newtheorem{lemma}[theorem]{Lemma} 
\newtheorem{corollary}[theorem]{Corollary} 
\newtheorem{proposition}[theorem]{Proposition} 

\theoremstyle{definition}
\newtheorem{definition}[theorem]{Definition} 
\newtheorem{question}[theorem]{Question}

\theoremstyle{remark}

\newtheorem{remark}[theorem]{Remark}

\newcounter{BenignClaim}

\declaretheorem[numberlike=BenignClaim,style=remark,name=Claim,refname={Claim,Claims}]{benignclaim}

\begin{document}
\maketitle
\begin{abstract}
SJT reducibility between sets $A,B \subseteq  \mathbb N$ is defined by $A \le_{SJT}  B$ if for each computable function $h$ that is unbounded and nondecreasing, there is an $h$-bounded uniformly $B$-c.e.\ trace $(T_n)_{n \in \mathbb N} $ such that for each $n$, the value $J^A(n)$ of the jump is in~$T_n$, if defined. This reducibility is slightly weaker than Turing reducibility. We study SJT reducibility, and as a main result give several characterisations of it on the $K$-trivial sets. This is the first case of extending the three lowness paradigms, weak as an oracle, computed by many, and inert, to the setting of weak reducibilities.
\end{abstract}

\section{Introduction}
The concept of a weak reducibility on  sets of natural numbers  arises by combining two general notions  developed in   computability theory from  the year 2002 onwards: lowness properties of sets~\cite{Nies:AM,Nies:ICM}, and partial relativisation~\cite{Barmpalias.Miller.ea:12}.   A \textit{lowness property} formalises a particular aspect of   computational  simplicity of a set of natural numbers (henceforth  simply called a set). Three paradigms have been proposed~\cite{Greenberg.Hirschfeldt.ea:12}:  being  weak as an oracle, computed by many oracles, and inert; see Section~\ref{ss:LP} below.  A \textit{partial relativisation} of a computability theoretic concept $\+ C$  is obtained by  relativising only certain   constitutents of the concept to an oracle, while the remaining constituents, typically bounds on sizes of finite sets or on the number of changes of a computable approximation of a set, are left unrelativised. See  \cite[Section~2.1]{Barmpalias.Miller.ea:12} for  examples.  It is customary to use the term  ``$\+ C$ by $X$" for  a  partial relativisation of $\+ C$  to an oracle $X$; the constituents of the concept that are relativised  are assumed to be understood from the context. 
Each of these two general notions was obtained by crystalising mathematical intuition that   developed  through     series of works;    the citations above  merely  point to some early  places (known to us) where the notions were formulated.  

\subsection{Weak reducibilities}  A  pre-ordering $\le_W$ on  sets is called a \emph{weak reducibility} if $\le_W$ is arithmetical, $X \leT Y$ implies $X \le_W Y$, and $X' \not \le_W X$ for each $X$.     This terminology \cite[Section~5.6]{Nies:book}  is chosen to be opposite to the terminology of   strong reducibilities  of Odifreddi~\cite{Odifreddi:81} (such as truth table   reducibility).

Given  an arithmetical  lowness property $\+ L$ and a proper choice of its partial relativisation, the binary relation ``$A $ is in $\+ L$ by $B$" tends to  be  a weak reducibility (in particular, it is transitive, and implied by $\leT$).  We consider two examples. 
The first is the much studied  LR-reducibility~\cite{Nies:AM}.  Let $\MLR^X$ be the class of \ML\ random sets relative to $X$, and recall that $A$ is low for ML-random if $\MLR \sub \MLR^A$. 
LR-reducibility $\le_{\LRR}$   is obtained by partially relativising this to $B$: for $A \le_{\LRR}B$ to hold, one requires that $\MLR^B\sub \MLR^A$, rather than the full relativisation to $B$, which would be $\MLR^B\sub \MLR^{A\oplus B}$. 

As a second  example, more  relevant to the present paper, let us consider the lowness property of being jump traceable.  Recall that an \emph{order function} is an unbounded, nondecreasing computable function on $\NN$.  A \emph{c.e.\ trace} is a uniformly c.e.\  sequence $\seq{T_n}\sN n$ of finite sets. For an order function $h$, such a trace   is $h$-bounded if $|T_n| \le h(n)$ for each $n$. A trace is computably bounded if it is $h$-bounded for some computable order function~$h$. 

\begin{definition}[\cite{Nies:06}] \label{def:jump_traceable}
A set $A$ is \emph{jump-traceable} if there is a computably bounded c.e.\ trace $\seq{T_n}\sN n$ such that for all~$n$, $J^A(n)$ is in $T_n$ if it is defined.	
\end{definition}
 Here $J^A$ is a universal $A$-partial computable function.  It is easy to see that each jump traceable set is $\mathrm{GL}_1$~\cite[8.4.3]{Nies:book}. The appropriate partial relativisation to an oracle $B$, namely $A$ is jump traceable by $B$, asks  that the trace be c.e.\ in $B$, but   still only  $J^A$, rather than  $J^{A\oplus B}$, is traced, and  the order function bounding the trace is still computable. This yields    the reducibility $\le_\mathrm{JT}$, introduced  in \cite[8.4.13]{Nies:book},  with transitivity verified in 8.4.14.  To check that  $\le_\mathrm{JT}$ is a weak reducibility, note that  it is clearly $\SI 4$, and implied by $\leT$.  Relativising the fact that each jump traceable is $\mathrm{GL}_1$ one shows that    $X' \not \le_\mathrm{JT}  X$ for each set $X$. 
 
The class of jump traceable sets  has a perfect $\PI 1$ subclass.  It  admits superhigh members~\cite{Kjos.Nies:09}, so it cannot be considered a   strong lowness property. If one requires that $J^A$ has an $h$-bounded trace for \emph{every} order function~$h$, one obtains the notion of strong jump traceability, introduced in~\cite{Figueira.ea:08}. This notion has been  studied extensively~\cite{Cholak.Downey.ea:08,Diamondstone.Greenberg.ea:15,DG:SJT2,Greenberg:11,Greenberg.Hirschfeldt.ea:12,Greenberg.Nies:11,Greenberg.Turetsky:14,Ng:10}, with applications outside randomness \cite{Downey.Greenberg:13}; see \cite{Greenberg.Turetsky:18} for a survey. The central notion for this paper is a corresponding weak reducibility.

\begin{definition}[SJT reducibility, \cite{Nies:book},  8.4.37] \label{def;SJT} For sets $A,B$, one writes $A \SJR B $ if for every order function $h$, there is a $B$-c.e., $h$-bounded trace $\seq{T_n}\sN n$ such that for each $n$,  if  $J^A(n)$  is defined then it is in~$T_n$.
\end{definition}
Note that a  set~$A$ is strongly jump traceable iff $A\SJR \emptyset$.  An argument similar to \cite[Theorem 3.3]{Ng:10} or \cite[8.4.14]{Nies:book} shows that the relation $\SJR$ is transitive.  
Clearly $X' \not \SJR X$  for each set $X$, because this already holds for $\le_\mathrm{JT}$. The relation $\SJR$  is  $\PI  5$. We also note that $\SJR$ is genuinely weaker than $\leT$, in the sense that for each set $B$ there is a set $C >_T B$ such that $C \SJR B$. To see this, by relativisation  of a result in~\cite{Figueira.ea:08} let $A$ be a set  such that $A \not \leT B$ and $A$ is  strongly jump traceable relative to $B$, and let $C= A\oplus B$.  Then for each order function $h$ relative to $B$, $J^C$ has a $B$-c.e.\ trace bounded by $h$. 
\begin{remark}
	A weak reducibility $\le_W$ induces the highness property of $\le_W$-hardness, namely the property of a set $A$  that $\Halt \le_W A$. This highness property  is in a sense dual to the lowness property that $\le_W$ is based on. Such highness properties have often been in the focus of research. Results  often state  coincidences between  highness properties; for instance,  a set is LR-hard iff it  is uniformly almost everywhere dominating~\cite{Kjos.Miller.ea:11}.  The  highness property of SJT-hardness was first studied     in \cite{Downey.Greenberg:13}, where the authors showed as their main result that  some incomputable, c.e.\ set is Turing below each c.e.\ SJT-hard set. 
\end{remark}

\subsection{Three lowness paradigms} \label{ss:LP}
As mentioned, three lowness paradigms for  a set $A\sub \NN$ have beeen proposed. We give a brief summary; see~\cite{Greenberg.Hirschfeldt.ea:12} for more detail.
\begin{enumerate}  \item    \textit{Weak as an oracle}: 
	\\ $A$ is not very useful as an oracle for Turing machines. 
	\item   \textit{Computed by many oracles}:  \\the collection of oracles computing $A$ is large. The class of such oracles    is  null   unless $A$ is computable, so largeness must be taken in a more specific sense. For instance, $A$ is  called a base for ML-randomness if there is  an oracle in $\MLR^A$ computing $A$.  
	\item    \textit{Inert}: \\ Recall that the Shoenfield limit lemma states  that a set  $A$ is $\DII$ iff it can be computably approximated with a finite number of changes. Inertness  means that there is such an approximation where the number of changes is in a sense small.  \end{enumerate}

 It is often interesting to ask whether a given lowness property has equivalent definitions according to   two or even all three  paradigms. We examine this for the property that a set $A$ is  low  for ML-randomness: every ML-random is ML-random relative to $A$.  \begin{enumerate}  \item  
  The  definition follows Paradigm~(1). Another equivalent characterisation according to this paradigm is lowness for $K$: there is a constant $d$ such that $K(x)\le K^A(x) + d$ for each string $x$, where $K$ denotes prefix free descriptive string complexity.
  \item  $A$ is low for ML-randomness iff  it is  a base for ML-randomness, which gives  a definition according to Paradigm~(2) (see \cite[Ch.\ 5]{Nies:book} for details and references).  
  \item  $A$ is low for ML-randomness iff it has a computable approximation with few changes  in the sense of the cost function $\cost_\Omega(x,s) = \Om_s- \Om_x$ by~\cite{Nies:17} ; this follows Paradigm~(3). An earlier version of this result for a  cost function~$\cost_K$ dates back to~\cite{Nies:AM}. \end{enumerate} We note that  lowness for ML-randomness also coincides with $K$-triviality,  a property expressing that the set is far from random~\cite{Downey.Hirschfeldt.ea:03,Nies:AM}.  
 
 Strong jump traceability  by its definition follows  Paradigm~(1). In~\cite{Greenberg.Hirschfeldt.ea:12,Diamondstone.Greenberg.ea:15} an equivalent characterisation following Paradigm~(2) is given: a set is strongly jump traceable if and only if it is computed by every $\omega$-c.a.\ (or every superlow, or every superhigh) ML-random set\footnote{Recall that a function $f\colon \NN\to \NN$ is \emph{$\omega$-computably approximable}, or $\omega$-c.a., if it has a computable approximation $f(x,s)$ such that $n\mapsto \#\left\{ s \,:\,  f(x,s)\ne f(x,s+1) \right\}$ is bounded by some computable function; equivalently,~$f\le_{\textup{wtt}} \emptyset'$. A set $X\subseteq \NN$ is $\omega$-c.a.\ if its characteristic function is~$\omega$-c.a. The terminology ``$\omega$-c.a.'' was introduced for clarity, since the term ``$\omega$-c.e.'' is often used for the class $\Sigma^{-1}_\omega$ in Ershov's effective difference hierarchy, which is a larger class of sets (the $\omega$-c.a.\ sets are those that are both $\omega$-c.e.\ and co-$\omega$-c.e.)}. A related result appears in \cite{Greenberg.Miller.ea:24}: a set is strongly jump-traceable if and only if it is computable from every infinite ``section'' of Chaitin's~$\Omega$. 
In \cite{Greenberg.Nies:11} it is shown that  a c.e.\ set is strongly jump traceable iff it obeys all 
benign cost functions; this yields a characterisation in terms of Paradigm (3).

\subsection{Characterising SJT-reducibility according to the three paradigms.} In this paper, we suggest to apply the three paradigms not only to lowness properties, but also to the corresponding weak reducibilities. Our main results show that, under certain restrictions, the characterisations of the strongly jump traceable sets according to the three paradigms can be lifted to the case of the weak reducibility~$\SJR$. 

The definition of $\SJR$ follows paradigm~(1), as does the definition of strong jump-traceability itself.  It indicates computational weakness of the oracle $A$ compared to $B$, because $J^A(n)$  is  caught in an ``arbitrarily small'' trace set enumerated by $B$. Our first result gives a characterisation according to Paradigm~(3): for every benign cost function $\cost$, the set $B$ computes an approximation of~$A$ for which the total cost of changes in the sense of~$\cost$ is finite.  For the precise definitions of the notation $A\models_B \cost$ and of benign cost functions, see \cref{def:cost_functions,def:benign} below.
\begin{theorem} \label{thm:cost_function_characterisation}
Suppose that both~$A$ and~$B$ are jump-traceable. The following are equivalent:
\begin{itemize}
\item[(a)] $A\SJR B$. 
\item[(b)] $A \models_B \cost$ for every benign cost function~$\cost$. 
\end{itemize}	
\end{theorem}

 We prove \cref{thm:cost_function_characterisation} in \cref{sec:benign_cost_functions}.

\medskip

Recall that $A \sub \NN$ is $K$-trivial if $K(A\uhr n) \lep K(n)$ for each $n$. There are strongly jump traceable sets outside the computable~\cite{Figueira.ea:08}, 
yet the class lies far inside the $K$-trivials~\cite{DG:SJT2,Greenberg.Miller.etal:19}. Recent work \cite{Greenberg.Miller.etal:19,Greenberg.Miller.ea:24} reveals a rich structure inside the $K$-trivial Turing degrees, based on both cost functions and computability from various random sets. This, and the existing relationship between strong jump-traceability and $K$-triviality, leads us to expect that the weak reducibility $\SJR$ will also shed new light on the $K$-trivial sets and degrees. Our second result gives a variety of characterisations of $\SJR$ within the $K$-trivial sets, according to Paradigm~(2). 
We quickly recall the definition of some  other classes that will be required.
\begin{itemize}
	\item A set~$Y$ is superhigh if $\emptyset'' \ltt Y'$. 
	\item A set~$Y$ is $\omega$-c.a.\ if its characteristic function is $\omega$-c.a.; equivalently if \mbox{$Y\lwtt \emptyset'$}. 
	\item A set~$Y$ is {superlow} if $Y'\ltt \emptyset'$; equivalently, if $Y'$ is $\omega$-c.a.
	\item Given an infinite set  $R\subseteq \NN$, let  $\Omega_R$ be the result of removing from~$\Omega$ all the bits in locations not in~$R$. As mentioned above, here $\Omega$ is any left-c.e.\ ML-random set. 
\end{itemize}
\begin{theorem} \label{thm:random_computing_characterisation}
Let $\+C$ be any of the following classes of sets:
\begin{itemize}
\item[(1)] sets that are not weakly Demuth random;
\item[(2)] superhigh sets;
\item[(3)] $\omega$-c.a.\ sets;
\item[(4)] superlow sets;
\item[(5)] the sets $\Omega_R$, where~$R$ is infinite and computable. 
\end{itemize}	
Then the following are equivalent for $K$-trivial sets~$A$ and~$B$:
\begin{itemize}
\item[(a)] $A\SJR B$. 
\item[(b)] $A \le_T B \oplus Y$ for every \ML-random set $Y\in \+C$. 
\end{itemize}	
\end{theorem}

Note that every $K$-trivial set is jump traceable. This was first shown by Nies~\cite{Nies:AM}, improving a result of  Zambella (see \cite{Terwijn:phd}). For weak Demuth randomness, see for example~\cite{Bienvenu.Downey.ea:14}; we give a detailed definition in \cref{sec:randoms_computing} (\cref{ref:weak_Demuth}).


\subsection{The difficulties with partial relativisation}

As the restrictions in the two theorems above indicate, lifting the various characterisations of strong jump traceability to the weak reducibility $\SJR$ is far from straightforward. Since arguments in computability theory almost always relativise, the difficulty lies in the fact that $\SJR$ is a partial relativisation. 

One of the difficulties is in modifying \emph{adaptive} arguments. In such an argument, we typically work with a strongly jump-traceable oracle~$A$, and define an $A$-partial computable function~$\psi$. By the recursion theorem, during the construction we have access to a c.e.\ trace~$\seq{T_n}\sN n$ for~$\psi$, and we use that trace to dynamically define~$\psi$: the definition of~$\psi$ \emph{adapts} to the enumerations appearing in $\seq{T_n}\sN n$. When $B\ne\emptyset$, more specifically, when~$A$ does not compute~$B$, the enumeration of $\seq{T_n}\sN n$  is only $B$-computable, and so~$A$ does not have such access to the trace. In a full relativisation to~$B$ we use $A\oplus B$ rather than~$A$, which allows us to use adaptive arguments. 

Similar difficulties occur in the study of other weak reducibilities such as $\le_{\LRR}$ mentioned above. However, our characterisation in \cref{thm:cost_function_characterisation} will show that this problem does not occur with $\SJR$ (under the same assumption on $A\oplus B$):

	\begin{theorem} \label{thm:A_to_A_join_B}
	Suppose that $A\oplus B$ is jump-traceable. The following are equivalent:
	\begin{itemize}
	\item[(a)] $A\SJR B$. 
	\item[(b)] $A\oplus B\SJR B$. 
	\end{itemize}	
\end{theorem}

We discuss \cref{thm:A_to_A_join_B}, and other corollaries of our characterisations, in \cref{sec:conclu}. Note that since Turing reducibility implies $\SJR$, $A\oplus B\SJR B$ if and only if $A\oplus B \equiv_{\textup{\scriptsize{SJT}}} B$. Regarding $\SJR$ and joins, see   \cref{sec:conclu}, in particular \cref{question:join} and the discussion thereafter.



\section{Benign cost functions}
\label{sec:benign_cost_functions}

We recall the definitions of cost functions and approximations obeying them; for background see \cite[Ch.\ 5]{Nies:book} or \cite{Nies:17}. 

\begin{definition} \label{def:cost_functions}
	A \emph{cost function} is a computable function $\cost \colon  \NN\times \NN \to \QQ^+$ satisfying:
	\begin{itemize}
		\item Monotonicity: $\cost$ is nonincreasing in the first variable and nondecreasing in the second. 
		\item The limit condition: for all~$x$, $\lim_s \cost(x,s)$ exists, and $\lim_x \lim_s \cost(x,s)=0$.
	\end{itemize}

	Let $\seq{A_s}$ be an approximation of a set~$A$. The \emph{total $\cost$-cost} of the approximation is the sum
	\[
		\sum_s \cost(x_s,s),
	\]
	where $x_s = |A_s\wedge A_{s+1}|$ is the least~$x$ such that $A_s(x)\ne A_{s+1}(x)$. 

	Let $A$ and~$B$ be sets. We write $A \models_B \cost$ if there is a $B$-computable approximation $\seq{A_s}\sN s$ of $A$ for which the total $\cost$-cost is finite. 
\end{definition}

One views $\cost(x,s)$ as the cost of changing an approximation at $x$ as the least element, at a stage $s$. The limit condition says that for large enough $x$, changing $A(x)$ at any stage is as cheap as one likes.

\begin{definition}[\cite{Greenberg.Nies:11}] 
\label{def:benign}
A cost function~$\cost$ is \emph{benign} if given a  rational  $\epsilon>0$, we can compute a bound on the length of any sequence $x_1 < s_1 \le x_2 < s_2 \le \cdots \le x_{\ell} < s_{\ell}$ such that $\cost(x_{i},s_{i})\ge \epsilon$ for all $i\le \ell$.
\end{definition}	

For example, the cost function $\cost_{\Omega}(x,s) = \Omega_s - \Omega_x$ is benign, with the bound for~$\epsilon$ being $1/\epsilon$. Here $(\Omega_s)$ is an increasing computable approximation of the left-c.e., ML-random sequence~$\Omega$.

\subsection{A proof of \cref{thm:A_to_A_join_B}} \label{ss:proof Thm 1.6}

We first explain how \cref{thm:A_to_A_join_B} follows from \cref{thm:cost_function_characterisation}. In one direction, suppose that $A\oplus B \SJR B$; since $A\le_{\textup{T}} A\oplus B$ and $\le_{\textup{T}}$ implies $\SJR$, we have $A\SJR B$. 

In the other direction, suppose that $A\oplus B$ is jump-traceable and that $A\SJR B$. The assumption on $A\oplus B$ implies that both~$A$ and~$B$ are jump-traceable, so \cref{thm:cost_function_characterisation} applies. Hence, $A\models_B \cost$ for every benign cost function~$\cost$. This implies that $A\oplus B\models_B \cost$ for every benign cost function~$\cost$: fix such a cost function~$\cost$. Suppose that $(A_s)$ is a $B$-computable approximation of~$A$ witnessing that $A\models_B \cost$. Let $C_s = A_s\oplus B$. The fact that $\cost$ is monotonic implies that the total $\cost$-cost of $(C_s)$ is bounded by the total $\cost$-cost of~$(A_s)$, and so is finite. 

The assumption on~$A\oplus B$, together with the other direction of \cref{thm:cost_function_characterisation}, now shows that $A\oplus B\SJR B$. 

\medskip

In the next two  sections  we give a proof of \cref{thm:cost_function_characterisation}.
\subsection{Proof of (b)$\RA$(a) of \cref{thm:cost_function_characterisation}}
In this direction, in fact, we do not make use of the assumption that~$B$ is jump-traceable. \cite[Prop.\ 2.1]{Greenberg.Nies:11} gives this direction for the case $B=\emptyset$, under the extra assumption that~$A$ is c.e.\ That proof does not partially relativise; however, we can now present a simpler argument that applies more generally. 

Suppose that~$A$ is jump-traceable, and that $A\models_B \cost$ for every benign cost function~$\cost$. Let~$h$ be an order function; we show that $J^A$ has a $B$-c.e.\ trace bounded by~$h$. We define an $A$-partial computable function~$\psi$ by setting $\psi(n) = \sigma$ if $J^A(n)\converge$ with use $\sigma\prec A$. Since~$A$ is jump-traceable, the partial function~$\psi$ has a c.e.\ trace~$\seq{V_n}\sN n$, bounded by some computable order function~$g$. 

We define a cost function~$\cost$ as follows: 
\[
	\cost(x,s) = \max \left\{ 1/h(n) \,:\,  \exists \sigma\in V_{n,s}\text{ s.t. }|\sigma|>x \right\}. 
\]
The function $\cost$ is monotonic: the fact that $V_{n,s}\subseteq V_{n,t}$ when $s<t$ implies that it is increasing in the second variable; it is clearly decreasing in the first variable. It satisfies the limit condition because $h$ is unbounded, and each~$V_n$ is finite. We show that~$\cost$ is benign. Let $\epsilon>0$. Suppose that $s< y \le t$ and that $\cost(y,t)\ge \epsilon$. There is some~$n$ with $1/h(n)\ge \epsilon$ and some $\sigma\in V_{n,t}$ with $|\sigma|>y$. Since $|\sigma|>s$ we may assume that $\sigma\notin V_{m,s}$ for any~$m$. Hence, the length of a sequence $x_1< s_1 \le x_2 < s_2 \le x_3 < s_3 \le \cdots$ with $\cost(x_i,s_i)\ge \epsilon$ is bounded by the size of
\[
	\bigcup \left\{ V_{n} \,:\,  1/h(n)\ge \epsilon \right\},
\]
which is bounded by 
\[
	\sum \left\{ g(n) \,:\,  1/h(n)\ge \epsilon \right\}. 
\]
Since~$h$ is computable, non-decreasing and unbounded, and~$g$ is computable, this sum is finite and can be obtained computably from~$\epsilon$. 

Now suppose that $(A_s)$ is a $B$-computable approximation whose total $\cost$-cost is finite; by ignoring finitely many stages of the approximation, we may assume that the total cost is bounded by~1. For each~$n$, let $T_n$ be the set of values $J_s^{A_s}(n)$ for~$s$ such that $J_s^{A_s}(n)\downarrow$ with use $\sigma\prec A_s$, and $\sigma\in V_{n,s}$.  If $J^A(n)\downarrow$ then since $\psi(n)\in V_n$, we have $J^A(n)\in T_n$. Let $s<t$ be two stages responsible for enumerating new values into~$T_n$; let $\sigma\prec A_s$ with $J_s^{A_s}(n) = J^\sigma(n)$ and $\s\in V_{n,s}$. By definition, $\cost(x,s)\ge 1/h(n)$ for every $x<|\sigma|$. Since $J_t^{A_t}(n)\ne J_s^{A_s}(n)$, we have $\sigma\nprec A_t$, so there is some stage $r$ with $s\le r <t$ and some $x<|\sigma|$ such that $A_r(x)\ne A_{r+1}(x)$. Thus, the contribution of stage~$r$ to the total $\cost$-cost of the approximation $(A_s)$ is at least $1/h(n)$. This shows that $|T_n|\le h(n)+1$. Replacing $h$ by $h-1$ gives the desired result.

\subsection{Proof of (a)$\RA$(b) of \cref{thm:cost_function_characterisation}}

We turn to the proof of the harder direction of \cref{thm:cost_function_characterisation}. The case $B=\emptyset$ is Theorem 1.12 of \cite{Diamondstone.Greenberg.ea:15}. That argument is adaptive; we overcome this by using the jump-traceability of~$B$. Very roughly, we define an $A$-partial computable function $\Phi^A$, which will have a $B$-c.e.\ trace~$\seq{U_x^B}\sN x$. Instead of directly defining values of $\Phi^A$ in reaction to elements being enumerated into the trace, we use the fact that~$B$ is jump-traceable to ``trace the trace''; we will have a c.e.\ trace~$\seq{V_y}\sN y$, with no oracle, that essentially traces values in~$\seq{U_x^B}\sN x$, and we define $\Phi^A$ in reaction to values showing up in~$\seq{V_y}\sN y$. The thing to note is that the size of~$\seq{U_x^B}\sN x$ will be bounded by a very slow-growing order function, making use of the assumption $A\SJR B$. However, we do not control the computable bound on the size of~$\seq{V_y}\sN y$; this bound can grow very quickly, as the assumption is only that $B$ is jump-traceable, not strongly so. Observe that in this direction, $A$ being jump-traceable follows from~$B$ being jump-traceable and $A\SJR B$, since the latter implies $A\le_{\textup{JT}} B$. 

\medskip

We will need a combinatorial lemma. 
\begin{lemma}\label{lem:binary_tree_size}
Suppose that $T \subset \omega^{<\omega}$ is a finite tree, and $v_0, \dots, v_{n-1} \in T$ are pairwise distinct. If each $v_i$ has at least two  children in $T$, then $T$ has at least $n+1$ leaves.
\end{lemma}
This can be proved by induction on $n$; we omit the details. The extreme case is when the nodes~$v_i$ form a chain (they are all comparable). 

\medskip

Fix a benign cost function $\cost$. There is an $\omega$-c.a.\ function  $f$ such that  $c(f(n)) < 2^{-n}$ for all $n$. There is a computable binary function, also denoted $f$, such that $\lim_s f(n, s) = f(n)$ for all $n$, and such that for each $n$, $|\{ f(n, s) : s \in \omega\}| \le g(n)$, where $g$ is some total computable function.  We may assume that $f(n,s)$ is non-decreasing in both $n$ and $s$. 

\medskip

Let us discuss the structure of the argument. It has two steps:
\begin{enumerate}
	\item[(1)] Enumeration of the functionals~$\Phi$ and~$\Psi$;
	\item[(2)] Using the results of~(1) to define the approximation of~$A$. 
\end{enumerate}
The first step is a computable construction, and as we later discuss, it has to be uniform in the indices for~$\Phi$ and~$\Psi$. The second step is computable in~$B$, and uses the correct indices of~$\Phi$ and~$\Psi$, so is non-uniform in that sense too. 

\smallskip

As discussed, the argument is adaptive, in that in the definition of $\Phi$ and~$\Psi$, we will make use of a trace~$\seq{V_y}\sN y$ that traces $\Psi^B$, and an oracle trace~$\seq{U_x^-}\sN x$ such that~$\seq{U_x^B}\sN x$ traces $\Phi^A$. Working with partial functions allows us to achieve this using the recursion theorem (as in \cite{Cholak.Downey.ea:08}), however the method of a universal trace would also work here (see for example \cite{Greenberg.Nies:11}).

\smallskip

The general idea (as in \cite{Diamondstone.Greenberg.ea:15}) is to \emph{test} potential initial segments of~$A$ by setting $\Phi^\s(x)=\s$ for all~$\s$ of some particular length~$\ell$, on a variety of locations (inputs)~$x$. When $B=\emptyset$, we can observe the result of such a test by seeing which~$\s$ are enumerated into $U_x$; these values are preferred to others. The test is more useful when the bound on $|U_x|$ is smaller, for example, when $|U_x|=1$ we know that the unique~$\s$ that appears in~$U_x$ must be an initial segment of~$A$. We will start by testing the length $\ell = f(n)$ on inputs~$x$ with $|U_x|\le n$. Of course, we do not know the final value $f(n)$, so we will need to test all lengths $f(n,s)$, for all~$s$, on such inputs~$x$. Once we test a certain length on~$x$, we cannot use the same input~$x$ again for other lengths, so we need to ensure that we have sufficiently many~$x$ with $|U_x|\le n$. Luckily, we control the bound on the size of~$U_x$, and we also know that we will encounter at most $g(n)$ many values $f(n,s)$, so this allows us to reserve sufficiently many~$x$ for this purpose. 

 The heart of the construction is a \emph{promotion} process for various lengths that we test. When $|U_x|>1$, we may get more than one result, $\s_0$ and~$\s_1$, both enumerated in $U_x$, and we need to devise a procedure that will tell us which of these we will believe more. In such a situation, we will (very roughly) promote the length $\ell = |\s_0| = |\s_1|$ from level~$n$ testing to level $(n-1)$-testing, which gives us stronger results (the actual details will be a bit different). To know that we have sufficiently many locations~$x$ for these extra tests, we need to bound the number of lengths promoted from level~$n$ to level $n-1$. To do this, some of our tests will not use a single location~$x$, but many such locations, and we will only trust some~$\s$ if it appears in \emph{all} $U_x$ on which it was tested. 

 This process of promotion is where the construction becomes adaptive: based on enumerations into various $U_x$, we decide to make further tests, defining more values of $\Phi^A$.  The overall structure of our construction, when $B\ne \emptyset$, is the same, except that, as discussed, we cannot examine $U_x^B$ during the construction. Rather, for any oracle~$Y$ (uniformly), we will define~$\Psi^Y$ while examining~$\seq{U_x^Y}\sN x$. As mentioned, the whole construction of~$\Phi$ and~$\Psi$ will also, adaptively, examine values in another trace~$\seq{V_y}\sN y$ (that has no oracle, and traces the ``correct'' $\Psi^B$). 

\smallskip

For the definition of~$\Psi$, we computably partition $\NN$ into sets $R_x$ (for $x\in \NN$) and $Q_n$ (for $n\in \NN$), such that $|R_x| = x$ and $|Q_n| = n$. We will shortly see how using the recursion theorem, during the construction, we have a c.e.\ trace~$\seq{V_y}\sN y$ for $\Psi^B$, and a computable bound~$h$ of~$\seq{V_y}\sN y$. 

Given~$h$, for the purpose of defining~$\Phi$, we now computably partition~$\NN$ to sets $I_n$ and~$J_n$, so that:
\[
	|I_n|  = g(n) + \sum \big\{ h(y)\,:\, y\in Q_{n+1}  \big\},
\]
and $|J_n| = 2^{r_n}$, where
\[
	r_n = \sum \Big\{ h(y) \,:\,  y\in \bigcup_{x\in I_n} R_x \Big\}. 
\]
\begin{benignclaim}[Number $o$ and trace~$\seq{U_x^-}\sN x$] \label{cl:witzi} During the construction, we have a number~$o$ and an oracle-c.e.\ trace~$\seq{U_x^-}\sN x$ such that for all $n>o$, for all $x\in I_n\cup J_n$, 
\begin{itemize}
	\item for all~$Y$, $|U_x^Y|\le n$; and
	\item If $\Phi^A(x)$ is defined, then $\Phi^A(x)\in U_x^B$. 
\end{itemize}
\end{benignclaim}
It would seem that~$\seq{V_y}\sN y$, $h$ and~$\seq{U_x^-}\sN x$ are provided by the recursion theorem in a straightforward way. However, we note that $h$ depends on the index of~$\Psi$, and that in general, the assumption $A\SJR B$ does not allow us to uniformly, given an order function~$k$, obtain a $k$-bounded trace of $J^A$ (using a universal trace we can obtain such a trace but which may omit tracing finitely many values of $\Phi^A$, which will amount to the same thing as we do here). We do the following.

We start with a c.e.\ trace of $J^B$, bounded by some order function $\tilde h$. For each~$e$, let $h^e(y) = \tilde h(e,y)$. Assuming that $J^Y(e,y) = \Phi_e^Y(y)$ for all~$Y$, $e$ and~$y$, this gives us, uniformly in~$e$, the index of a c.e.\ trace~$\seq{V_y^e}\sN y$ of $\Psi_e^B$ bounded by~$h^e$. 

The definition above then gives us, for each~$e$, a partition of $\NN$ to sets $I_n^e$ and $J_n^e$ ($n \in \NN$), using the bound~$h^e$. This allows us to compute an order function~$k$ such that for all $d$, $e$, $n > \seq{d,e}$ and $x\in I^e_n\cup J^e_n$, we have $k(d,x)\le n$. Fix a $k$-bounded, $B$-c.e.\ trace~$\seq{\tilde{U}_x^B}\sN x$ of $J^A$; we may assume that for any oracle~$Y$,~$\seq{\tilde{U}_x^Y}\sN x$ is $k$-bounded. This allows us to compute, uniformly in~$d$ and~$e$, a c.e.\ index of an oracle trace~$\seq{U_x^-}\sN x$ that satisfies the required conditions when $\Psi=\Psi_e$ and $\Phi = \Phi_d$, where $ o = \seq{d,e}$. 

By the recursion theorem, during the construction we know the indices~$d$ and~$e$, and so the number~$o$, the sets $I_n$ and $J_n$, and c.e.\ indices of~$\seq{V_y}\sN y$ and~$\seq{U_x^-}\sN x$, as promised.  This shows \cref{cl:witzi}.

\medskip

We are ready to give the details of the construction of $\Phi$ and~$\Psi$. The definition of~$\Psi$ on the sets $R_x$ is immediate:
\begin{itemize}
	\item[(i)] For all $n>o$, for all $x\in I_n$, for every oracle~$Y$, for all $i<n$, if~$\ell$ is the $i$th number enumerated into $U_x^Y$, and $y$ is the $i$th element of~$R_x$, then we define $\Psi^Y(y) = \ell$. 
\end{itemize}
Note that we may assume that $x\in I_n$ implies $x\ge n$, so setting $|R_x| = x$ is certainly sufficient to make these definitions, recalling that $|U_x^Y|\le n$ when $n>0$ and $x\in I_n$. 

We will have a module for each $n > o$. This module can perform the following actions:
\begin{itemize}
	\item Test a length on~$I_n$; 
	\item Test a string on~$J_n$; 
	\item For any oracle~$Y$, declare that~$Y$ believes that a length~$\ell$ should be promoted;
	\item Promote some lengths to the $(n-1)$-module. 
\end{itemize}

We describe what each of these means, when these actions are performed, and argue that we can indeed perform these actions.

\subsubsection*{Promoting lengths} Let $n>o$. 

\begin{itemize}
	\item[(ii)] The $n$-module promotes every $\ell\in \bigcup_{y\in Q_n} V_y$. 
\end{itemize}

This definition applies even to the smallest module $n = o+1$, however, for this module, no actual action will be taken. 

\subsubsection*{Testing lengths on~$I_n$}

\begin{itemize}
	\item[(iii)] At stage~$0$, the $n$-module tests $f(n,0)$ on~$I_n$. At stage $s>0$, if $f(n,s)\ne f(n,s-1)$, the $n$-module tests $f(n,s)$ on~$I_n$. Finally, if at some stage~$s$, the $(n+1)$-module promotes a length~$\ell$, then the $n$-modules tests~$\ell$ on~$I_n$. 

	\item[(iv)] Testing a length~$\ell$ on~$I_n$ at stage~$s$ means: choose some unused $x\in I_n$, and define $\Psi^\s(x)=\s$ for all binary strings~$\s$ of length~$\ell$. 
\end{itemize}

We need to argue that such a fresh~$x$ can always be chosen; that is, we need to show that the $n$-module will test at most $|I_n|$ many lengths during the entire construction. The number~$g(n)$ bounds the number of lengths $f(n,s)$ ever tested on~$I_n$. By~(ii), at most $\sum \left\{ h(y) \,:\,  y\in Q_{n+1} \right\}$ many lengths are every promoted by the $(n+1)$-module. So the definition of~$|I_n|$ ensures that it is large enough. 

\subsubsection*{Testing strings on~$J_n$}

\begin{itemize}
	\item[(v)] The $n$-module tests on~$J_n$ every string in $\bigcup_{x\in I_n} \bigcup_{y\in R_x} V_y$. 
\end{itemize}

Note that the $n$-module will test at most $r_n$ many strings during the construction. Since $|J_n| = 2^{r_n}$, for the purposes of such testing, we identify $J_n$ with $\{0,1\}^{r_n}$, which we think of as a hypercube of side length 2 and dimension $r_n$. We let $d=0,1,\dots, r_n-1$ denote the axes of this hypercube; for $d<r_n$ we let $J_n(d) = \{ \tau \in J_n : \tau(d) = 0\}$. This is one half of a split of the hypercube into two pieces orthogonal to axis~$d$. 

\begin{itemize}
	\item[(vi)] Testing a string~$\s$ on~$J_n$ means the following.  We choose an unused axis~$d$  and define $\Phi^\sigma(x) = \sigma$ for all unused $x\in J_n(d)$.
\end{itemize}

Now~(v) ensures that indeed, there will always be an unused axis to choose. 

\subsubsection*{Seeking promotion}
Let $n>o$ and let~$Y$ be an oracle. 

\begin{itemize}
	\item[(vii)] We say that~$Y$ \emph{confirms} a string~$\s$ tested on~$J_n$ if $\s\in U_x^Y$ for every $x\in J_n$ on which we defined $\Phi^\s(x)=\s$. 

	\item[(viii)] Let $s$ be a stage; let~$\ell$ be the longest length which~$Y$ believed should be promoted prior to stage~$s$. Suppose that there are two strings, $\s_0$ and~$\s_1$, of the same length, such that:
	\begin{itemize}
		\item $\s_0\rest{\ell} = \s_1\rest{\ell}$; and
		\item $Y$ confirms both~$\s_0$ and~$\s_1$ on~$J_n$ at stage~$s$. 
	\end{itemize}
	Then we say that~$Y$ believes that the length $|\s_0|= |\s_1|$ should be promoted by the $n$-module. 

	\item[(ix)] The action taken then is: if $m$ is the $i$th length that $Y$ believes should be promoted by the $n$-module, and $i<n$, and $y$ is the $i$th element of~$Q_n$, then we define $\Psi^Y(y) = m$. 
\end{itemize}

We remark that it is quite possible that some oracles will believe that more than~$n$ many lengths should be promoted by the $n$-module. However this will not be the case for $Y=B$, as the next claim states. 

We presented the construction in a way that makes it easier to show that it can indeed be performed. It would be good, though, to follow the steps in order: 
\begin{enumerate}
\item A length $\ell = f(n,s)$ is tested on~$I_n$; this defines $\Phi$ on an element of~$I_n$ ((iii) and (iv) above).
\item This prompts strings of length~$\ell$ to appear in $U_x^Y$, and we make a definition of $\Psi$ on an element of $y\in R_x$ (i).
\item This prompts some strings to appear in $V_y$; in turn, these strings are tested on~$J_n$ ((v) and (vi)).
\item This prompts strings to show up in various $U_x^Y$ (for $x\in J_n$), and $Y$ believes that some length~$\ell$ should be promoted; a definition of $\Psi^Y(y)$ for some $y\in Q_n$ is made (with value~$\ell$, not a string of length~$\ell$) ((vii), (viii) and (ix)).
\item This prompts lengths to show up in $V_y$, and the $n$-module promotes them (ii). We are now back to step~(1) for the $(n-1)$-module for this promoted length. 
\end{enumerate}

This completes the description of the construction of $\Phi$ and~$\Psi$. 

\begin{benignclaim}
For each $n > o$, $B$ believes at most $n-1$ lengths should be promoted by the $n$-module.
\end{benignclaim}

\begin{proof}
Suppose not, and fix lengths $0 = \ell_0 < \ell_1 <  \dots < \ell_n$ such that for $i > 0$, $B$ believes $\ell_i$ is to be promoted by the $n$ module.  For each $i > 0$, fix strings $\sigma_0^i$ and $\sigma_1^i$ on the basis of which $B$ believes $\ell_i$ should be promoted by the $n$ module.

By construction, $B$ believes $\ell_i$ should be promoted before it believes $\ell_{i+1}$ should be, and thus $\sigma_0^{i+1}\uhr{\ell_i} = \sigma_1^{i+1}\uhr{\ell_i}$ for $i > 0$.  Clearly this also holds for $i = 0$.  Now define the following sequence of sets (note that our subscripts are decreasing):
\begin{itemize}
\item $Z_n = \{ \sigma_0^{n}, \sigma_1^{n}\}$;
\item For $0 < i < n$, 
\[
Z_i = Z_{i+1} \cup ( \{\sigma_0^i, \sigma_1^i\} \setminus \{ \sigma\uhr{\ell_i} : \sigma \in Z_{i+1}\}).
\]
\end{itemize}
Note that each $Z_i$ is an antichain, and $\{\sigma_0^i, \sigma_1^i\} \subseteq \{ \sigma \uhr{\ell_i} : \sigma \in Z_i\}$ by construction.  Let $T = \{ \sigma\uhr{\ell_i} : \sigma \in Z_1 \ \& \ i \le n\}$, which we think of as a tree.  The leaves of $T$ are precisely the elements of~$Z_1$.

For $i < n$, let $v_i = \sigma_0^{i+1}\uhr{\ell_i} = \sigma_1^{i+1}\uhr{\ell_i}$.  Then the $v_i$ are pairwise distinct because they all have different lengths, and each has at least 2 children in $T$ (namely, $\sigma_0^{i+1}$ and $\sigma_1^{i+1}$). By Lemma~\ref{lem:binary_tree_size}, $|Z_1| \ge n+1$.

Let $D$ be the set of axes chosen for various $\sigma \in Z_1$, and define $\tau \in J_n$ by
\[
\tau(d) = \left\{\begin{array}{cl}
0 & d \in D,\\
1 & d \not \in D.
\end{array}\right.
\]
Observe that $\tau \in J_n(d) \iff d \in D$.

Recall that our construction will define $\Phi^\sigma(\tau) = \sigma$ for all $\sigma \in Z_1$ unless $\Phi^\sigma(\tau)$ was already defined to be something else.  But no $\sigma' \not \in Z_1$ could have caused such a previous definition, as such $\sigma'$ will have an axis $d \not \in D$, and so will not seek to make a definition at $\tau$.  And $\sigma' \in Z_1 \setminus \{\sigma\}$ will not do this, as they will only seek to make a definition for $\Phi^{\sigma'}(\tau)$, and $\sigma'$ and $\sigma$ are incomparable as $Z_1$ is an antichain.

Thus $\Phi^\sigma(\tau) = \sigma$ for all $\sigma \in Z_1$.  And by the definition of promotion and confirmation, we have $Z_1 \subseteq U_\tau^B$, contradicting $|U_\tau^B| \le  n$.
\end{proof}

It follows that for every length $\ell$ which $B$ believes should be promoted by the $n$-module, there is a corresponding $y\in Q_n$ such that $\Psi^B(y)= \ell$. Since~$\seq{V_y}\sN y$ traces $\Psi^B$, $\ell\in V_y$. Thus, every length which $B$ believes should be promoted is eventually promoted.

\smallskip

We now proceed to the second part of the argument, in which we fix $Y=B$; the following construction is computable relative to~$B$. 

Let $\hat{\ell}$ be the longest length which $B$ believes the $o+1$ module should promote.  Nonuniformly, fix $A\uhr{\hat{\ell}}$.  Let $L(n,s)$ be the set of lengths being tested by the $n$ module at stage $s$ (tested on~$I_n$).  At a stage $s$, define a partial sequence $\sigma_n^s$ for $o \le n \le s$ recursively:
\begin{itemize}
\item $\sigma_o^s = A\uhr{\hat{\ell}}$;
\item Given $\sigma_n^s$, if there is some length which $B$ believes should be promoted by the $n+1$ module at stage $s$, but which is not yet promoted, then leave $\sigma_{n+1}^s$ undefined.  Otherwise, define $\sigma_{n+1}^s$ to be a string $\rho$ extending $\sigma_n^s$ with $|\rho| = \max L(n+1,s)$, and such that for each $\ell \in L(n+1, s)$, $\rho\uhr{\ell}$ is confirmed by $B$ at $n+1$ by stage $s$, if such a string $\rho$ exists.
\end{itemize}

\begin{benignclaim}
There is at most one possible choice for $\sigma_n^s$.
\end{benignclaim}

\begin{proof}
For $n = o$, this is immediate.

For $n > o$, suppose there were two distinct strings $\rho_0$ and $\rho_1$ which are appropriate to pick for $\sigma_n^s$.  Fix $\ell \in L(n, s)$ least with $\rho_0\uhr{\ell} \neq \rho_1\uhr{\ell}$.  Then $\rho_0\uhr{\ell}, \rho_1\uhr{\ell}$ witness that $B$ believes $\ell$ should be promoted at stage $s$, and $\ell > |\sigma_{n-1}^s|$, as $\rho_0$ and $\rho_1$ both extend $\sigma_{n-1}^s$.  This contradicts $|\sigma_{n-1}^s| = \max L(n-1, s)$ (or contradicts the definition of $\hat{\ell}$ if $n = o+1$).
\end{proof}

\begin{benignclaim}
Let $\ell_n = \max \bigcup_s L(n,s)$ for $n > o+1$, and $\ell_{o+1} = \hat{\ell}$.  Then $A\uhr{\ell_n} = \lim_s \sigma_n^s$ for $n > o$.
\end{benignclaim}

\begin{proof}
Induction on $n$.  The case $n = o+1$ is immediate.

For $n > o$, first observe that $\ell_{n-1}$ is either a length promoted by the $n$ module (and so eventually an element of $L(n,s)$) or is $f(n-1, s) \le f(n,s)$ for some $s$, and so is bounded by an element of $L(n,s)$.  Thus $\ell_{n-1} \le \ell_n$.

Now fix $s_0$ sufficiently large such that $\sigma_m^s = A\uhr{\ell_m}$ for all $m < n$ and $s \ge s_0$, and such that $L(n,s_0) = \bigcup_s L(n,s)$.  As $\Phi^A$ is traced by~$\seq{U_x^B}\sN x$, there is a stage $s_1 \ge s_0$ such that each $A\uhr{\ell}$ for $\ell \in L(n, s_0)$ is confirmed at $n$.  Then $A\uhr{\ell_n}$ is a possible choice for $\sigma_n^s$ for every $s \ge s_1$, and thus is $\sigma_n^s$.
\end{proof}

Define a sequence of stages $(s_t)_{t \in \omega}$ as follows:
\begin{itemize}
\item $s_0 = o$.
\item Given $s_t$, $s_{t+1}$ is the least $s>s_t$ such that for every $n$ with $o < n \le t$, $\sigma_n^s$ exists. 
\end{itemize}
Define $A_t = \sigma_t^{s_t}$.

\begin{benignclaim}
The total $\cost$-cost of the approximation $\seq{A_t}$ is finite. 
\end{benignclaim}

\begin{proof}
Suppose that $o < n \le t$ and $A_t(z) \neq A_t(z+1)$ for some $z$ with $c(z, t) \ge 2^{-n}$.  
As $c(z, s_t) \ge c(z, t)$, $z < f(n, s_t) \in L(n, s_t)$.  Thus $\sigma_n^{s_t} \neq \sigma_n^{s_{t+1}}$.  Fix $m$ least with $\sigma_m^{s_t} \neq \sigma_m^{s_{t+1}}$.  Fix $\ell \in L(m, s_t)$ least with $\sigma_m^{s_t}\uhr{\ell} \neq \sigma_m^{s_{t+1}}\uhr{\ell}$.   If no length less than $\ell$ and greater than $\max L(m-1, s_t)$ is promoted by the $m$ module at a stage $s \in (s_t, s_{t+1}]$, then these witness the promotion of $\ell$ at stage $s_{t+1}$, and $\ell > \max L(m-1, s_t)$, as both $\sigma_m^{s_t}$ and $\sigma_m^{s_{t+1}}$ extend $\sigma_{m-1}^{s_t} = \sigma_{m-1}^{s_{t+1}}$.  So whenever there is such a $z, n$ and $t$, there is a promotion by an $m$ module for $m \le n$ at a stage after $s_t$.

There can be at most $\sum_{o < m \le n} m-1 < n^2$ such promotions over the entire construction.  Thus we can bound the total $\cost$-cost of $\seq{A_t}$ by
\[
 \sum_n n^2 \cdot 2^{-n} < \infty.\qedhere
\]
\end{proof}
This completes the proof of \cref{thm:cost_function_characterisation}.


\section{Computing with random sets}
\label{sec:randoms_computing}

In this section we give a proof of \cref{thm:random_computing_characterisation}. First we discuss the notions of Demuth randomness and weak Demuth randomness; the latter appears in the statement of the theorem, while  the former will play a role in its proof.  We have the implications 
\bc Demuth random $\RA $ weakly Demuth random $\RA$ ML-random.  \ec
 For further background,    see \cite[Section 3.6]{Nies:book} on Demuth randomness, and  in 
particular why this property is  still compatible with being $\DII$. See~\cite{Bienvenu.Downey.ea:14}   for weak Demuth randomness. 

The general idea is that a Demuth test is a sequence $(G_m)$ of effectively open ($\Sigma^0_1$) sets with $\leb G_m \le 2^{-n}$, but unlike a ML-test, the sets $G_m$ are not \emph{uniformly}~c.e. Rather, we think of the sets being given to us in stages, where from time to time, we are allowed to empty a component of the test and restart its enumeration. Each component is restarted only finitely many times, and in fact we require that there is a computable function bounding the number of times each component is restarted. For the formal definition, recall again that a function $f\colon \NN\to\NN$ is $\omega$-c.a.\ if it has a computable approximation $f(m,s)$ such that the number of mind-changes $\# \{s\,:\, f(m,s)\ne f(m,s+1)\}$ is bounded by~$g(m)$, where~$g$ is some computable function.

\begin{definition} \label{def:Demuth}  Let $(W_e)$ be an admissible listing of the effectively open subsets of Cantor space. A \emph{Demuth test} is a sequence  $\seq {G_m}\sN m$ satisfying:
\begin{itemize}
	\item For all~$m$, $\leb G_m \le 2^{-m}$; and
	\item There is an $\omega$-c.a.\ function~$p$ such that $G_m = W_{p(m)}$ for all~$m$. 
\end{itemize}
A set $Y\in 2^{\NN}$ is \emph{captured} by a Demuth test $(G_m)$ if $Y\in G_m$ for infinitely many~$m$. Otherwise, we say that~$Y$ \emph{passes} the test.  A set is \emph{Demuth random} if it passes all Demuth tests.   
\end{definition}

If $(G_m)$ is a Demuth test, witnessed by an $\omega$-c.a.\ function~$p$, and $p(m,s)$ is a computable approximation of~$p$, then we write $G_m[s]$ for the clopen set $W_{p(m,s),s}$: at stage~$s$, we guess that $G_m = W_{p(m,s)}$, and we let $G_m[s]$ be the result of enumerating that effectively open set for~$s$ many steps. We may assume that for all~$m$ and~$s$, $\leb G_m[s]\le 2^{-m}$: we can stop enumerating clopen subsets into the component if we see that its measure will exceed $2^{-m}$. Note that in that case, we know that $p(m,s)\ne p(m)$, so we can wait for a new version of~$G_m$ to be started. 

Thus, an ``index-free'' equivalent definition of a Demuth test $(G_m)$ is: there is a computable array $(G_{m,s})$ of clopen sets such that for all~$m$ and~$s$, $\leb G_{m,s}\le 2^{-m}$, and there is a computable function~$g$ such that for all~$m$, $\# \{s\,:\, G_{m,s}\nsubseteq G_{m,s+1}\}$ is bounded by $g(m)$; and $G_m = \bigcup_{s>t(m)} G_{m,s}$, where $t= t(m)$ is any stage sufficiently late so that for all $s\ge t$, $G_{m,s}\subseteq G_{m,s+1}$. 

\smallskip

Note that the capturing condition is the Solovay one, rather than requiring $Y\in \bigcap_m G_m$. This is because in \cref{def:Demuth}, we do not require that the test be \emph{nested}, meaning $G_{m+1}\subseteq G_m$ for all~$m$. This is not an issue when considering ML-randomness; every non-nested ML test $(U_m)$ can be transformed into a nested one by replacing $U_m$ by $\bigcup_{n>m} U_n$. Applying this transformation to a Demuth test will usually not result in a Demuth test. 

\begin{definition} \label{ref:weak_Demuth}
	A set~$Y$ is \emph{weakly Demuth random} if it passes every nested Demuth test.   
\end{definition}

An equivalent definition is: for every Demuth test $(G_m)$ (nested or not), $Y\notin \bigcap_m G_m$. That is, rather than restricting to nested tests, we can replace the passing condition. To see this, given a Demuth test~$(G_m)$, we let $H_m = \bigcap_{n\le m} G_n$; then $(H_m)$ is a nested Demuth test that captures all $Y\in \bigcap_m G_m$. 

\medskip

\subsection{A proof of (a)$\RA$(b) of \cref{thm:random_computing_characterisation}}

In this direction, we only need the assumption that~$B$ is $K$-trivial. For each one of the five classes $\+C$ listed in \cref{thm:random_computing_characterisation}, let (b)$_{\+C}$ indicate item~(b) of the theorem as applied to this class: the statement that for all ML-random $Y\in \+C$, $A\le_\Tur Y\oplus B$. Of course, if $\+C\subseteq \+D$ then (b)$_{\+D}$ implies (b)$_{\+C}$. We note that among the five classes, the first,   the class of sets that are not weakly Demuth random,  is the largest. Namely, a weakly Demuth random set cannot be superhigh (this is a result of \Kuc\ and Nies \cite[Cor.\ 3.6]{Kucera.Nies:11}) and cannot be $\omega$-c.a.\ (this is essentially by definition; if $Y$ is $\omega$-c.a.\ then $(\{Y\uhr{m}\})$ is a nested Demuth test capturing~$Y$).  We note that every superlow set is $\omega$-c.a., and that every set~$\Omega_R$ for an infinite computable~$R$ is $\omega$-c.a.\ as well. 

Thus, in this direction, it suffices to show that (a) implies (b)$_{\+C}$ where $\+C$ is the class of sets that are not weakly Demuth random. 

\medskip

Hirschfeldt and Miller showed that for each null $\PI 2$ class $\mathcal H\sub \cantor$,  there is a cost function $\cost$ such that for each $A, Y \in \cantor$, if $A \models \cost$ and $Y \in \MLR \cap \mathcal H$ then  $A \le_T Y$; see \cite[proof of 5.3.15]{Nies:book} for a proof of  this otherwise unpublished result. For the direction under discussion, we need to show that when $\+H$ is the class of sets captured by some nested Demuth test, then the associated cost function~$\cost$ is benign, and that the Hirschfeldt-Miller result partially relativises as necessary. 
(See also the post on weak Demuth randomness by \Kuc\  and Nies in \cite{LogicBlog:11}.)

\smallskip

Let $(G_m)$ be a nested Demuth test, with approximation $(G_{m,s})$ as discussed above. Note that we may assume that for each~$m$ and~$s$, $G_{m+1,s}\subseteq G_{m,s}$. For $k<t$ we let:
\begin{itemize}
	\item $r(k,t)$ be the smallest~$m$ such that for some~$s$ with $k<s\le t$, $G_{m,s-1}\nsubseteq G_{m,s}$ (the smallest~$m$ such that the $m$th component of the test was restarted at some stage between~$k$ and~$t$). 

	\item $V_{k,t}=  \bigcup_{k < s \le t} G_{r(k,s),s}$.

	\item $\cost(k,t) = \leb V_{k,t}$. 
\end{itemize}

By definition, $V_{k,t}\subseteq V_{k,t+1}$, so~$\cost$ is increasing in the second variable. Also, $V_{k,t}\supseteq V_{k+1,t}$: we have $r(k,s)\le r(k+1,s)$, and so $G_{r(k,s),s}\supseteq G_{r(k+1,s),s}$. Hence, $\cost$ is decreasing in the first variable, so~$\cost$ is monotonic. 

Note that if no component $G_{m'}$ for $m'\le m$ is restarted after stage~$k$, then $r(k,s)>m$ for all $s>k$, implying that $G_{r(k,s),s}\subseteq G_{m,s}$ for all $s>k$, and so $V_{k,t}\subseteq G_m$ for all $t>k$.  

Suppose that $\cost(k,t)>2^{-m}$. Since $\leb G_m\le 2^{-m}$, this implies that $V_{k,t}\nsubseteq G_m$, so some component~$m'$ is restarted between stages~$k$ and~$t$. Since there is a computable bound on the number of times each component is restarted, we see that the cost function $\cost$ is benign.

We also observe that for all~$k$, $\bigcap_m G_m \subseteq \bigcup_{t>k} V_{k,t}$. This is because $\lim_t r(k,t) = m$ where~$m$ is smallest such that the $m$th component is restarted after stage~$k$, and so $G_m\subseteq \bigcup_{t>k} V_{k,t}$. 

Suppose now that $A \models_B \cost$; let $(A_s)$ be a $B$-computable approximation with a finite  total $\cost$-cost. For each stage~$s$, let $x_s$ be the least such that $A_s(x)\ne A_{s+1}(x)$; we let
\[
	R_s = V_{x_s,s}. 
\]
Each set~$R_s$ is clopen, and the sequence $(R_s)$ is computable in~$B$. Since $\cost(x_s,s) = \leb V_{x_s,s} = \leb R_s$, we get $\sum_s \leb R_s < \infty$. This means that $(R_s)$ is a \emph{Solovay test} relative to~$B$. 

Suppose that $Y$ is ML-random and is captured by the nested Demuth test~$(G_m)$. Since~$B$ is assumed to be $K$-trivial, it is low for ML-random, and so~$Y$ is ML-random relative to~$B$. Since $(R_s)$ is a Solovay test relative to~$B$, $Y$ cannot be captured by~$(R_s)$, meaning that $Y\in R_s$ for only finitely many~$s$. (Indeed, recall that a set is ML-random relative to~$B$ if and only if it passes all Solovay tests relative to~$B$.)

Let $s_0$ be such that $Y\notin R_s$ for all $s\ge s_0$. We now show how to compute~$A$ given $Y\oplus B$. We are given some~$x$ and we find~$A(x)$. To do that, using the oracle~$Y$, find some $t\ge s_0$ such that $Y\in V_{k,t}$; we observed above that such~$t$ exists, since $Y\in \bigcap_m G_m$. We claim that $A(x)= A_t(x)$ (here the oracle~$B$ is used, since the approximation $(A_t)$ is $B$-computable). Otherwise, $A_s(x)\ne A_{s+1}(x)$ for some $s\ge t$; so $R_s = V_{y,s}$ for some $y\le x$. Since $V_{y,s}\supseteq V_{x,s}\supseteq V_{x,t}$, $Y\in R_s$, contrary to the assumption on~$s_0$.

\subsection{A proof of (b)$\RA$(a) of \cref{thm:random_computing_characterisation}, cases (1)--(4).}

In this subsection we prove that (b)$_{\+ C}$ implies~(a) for each of the first four classes $\+C$ mentioned in \cref{thm:random_computing_characterisation}: not weakly Demuth random, superhigh, $\omega$-c.a., and superlow. The common property of these classes that is relevant here is that they are all Demuth-compatible, as defined in Nies~\cite{Nies:11}:

\begin{definition} \label{def:Demuth_compatible}
	A class $\+C\subseteq  2^{\NN}$ is \emph{Demuth-compatible} if for every Demuth test, there is some $Y\in \+C$ that passes this test. 
\end{definition}

Using some methods from \cite{Greenberg.Hirschfeldt.ea:12}, Nies showed in \cite[Section 4]{Nies:11} that the class of superlow sets, and the class of superhigh sets, are Demuth-compatible. Every superlow set is $\omega$-c.a., so the class of $\omega$-c.a.\ sets is Demuth-compatible; every $\omega$-c.a.\ is not weakly Demuth random, so the collection of sets that are not weakly Demuth random is also Demuth-compatible. 

Thus, for the four classes under consideration, it suffices to show that if $\+C$ is any Demuth-compatible class, then (b)$_{\+C}$ implies~(a). The argument combines methods from \cite{Nies:11} and \cite{Bienvenu.Downey.ea:14}. 

We remark that for this implication, we can relax the requirement that~$A$ and~$B$ be $K$-trivial:   it suffices to assume that $A$ and~$B$ are each computable from some c.e., jump-traceable (equivalently, superlow by \cite{Nies:06})  set.  This is because every $K$-trivial set is jump-traceable, and every $K$-trivial set is computable from some c.e.\ $K$-trivial set.     \cite[Cor.\:2.4]{Kjos.Nies:09} implies the following fact that we will use in the proof of the implication:  if $X$ is computable from a c.e., jump-traceable set, then $X$ is \emph{low for $\BLR{}$}, as recalled next.
 \begin{definition}[Cole and Simpson \cite{Cole.Simpson:07}] \label{def:BLR} \label{def:CS}
   	Let $B$ be an oracle. A function $f$ is \emph{bounded limit-recursive} in~$B$, denoted $\BLR(B)$, if it has a $B$-computable approximation $f(x,s)$, for which the number of mind-changes $\#\left\{ s \,:\,  f(x,s)\ne f(x,s+1) \right\}$ is bounded by a computable function.  
\end{definition}
We note that being bounded limit-recursive is a partial relativisation of the notion of $\omega$-c.a.; the full relativisation would allow a $B$-computable bound on the number of mind-changes.  
\begin{definition} We say that an oracle $X$  is \emph{low for $\BLR{}$} if  the $\BLR(X)$ functions are precisely the $\omega$-c.a.\ functions. 
\end{definition}

Associated with the notion of partial relativisation of being $\omega$-c.a.\ in \cref{def:CS} is a partial relativisation of Demuth tests. A $\DemBLR \la B \ra$ test (\cite[Def.\ 1.7]{Bienvenu.Downey.ea:14}) is defined like a Demuth test relative to~$B$, except that the function giving the $\Sigma^0_1(B)$ index of the $m$th component is $\BLR(B)$, rather than $\omega$-c.a.\ relative to~$B$. In terms of approximation, a $\DemBLR \la B \ra$ test is the limit of a $B$-computable approximation $(G_{m,s})$, with a \emph{computable} bound on the number of times the $m$th component is restarted, rather than a $B$-computable one. 

\medskip

Fix an order function~$h$, and sets~$A$ and~$B$ as described. We prove:
\begin{itemize}
	\item[$(*)$] There is a Demuth test as follows.  If $A\le_\Tur Y\oplus B$ for some $Y$ that passes the test, then $J^A$ has an $h$-bounded, $B$-c.e.\ trace. 
\end{itemize}
This suffices to show the desired implications.

\medskip

We will show~($*$) in a couple of steps. First, we will only consider a single Turing functional (reduction procedure) witnessing $A\le_\Tur Y\oplus B$. We will show that there is a $\DemBLR \la B \ra$ test as required, and then cover it by an unrelativised Demuth test. Finally, we will use a universal functional to obtain~$(*)$. 

For the time being, fix a Turing functional~$\Psi$. 

\medskip

We start by mostly following the proof of \cite[Thm.\:3.2]{Nies:11}. For $m\in \NN$ let $I_m = \left\{ x\in \NN \,:\,  2^m\le h(x) < 2^{m+1} \right\}$, so each~$I_m$ is a finite interval and together they partition~$\NN$. For $\sigma\in 2^{<\NN}$ let 
\[
	U_\sigma = \left\{ Z\in 2^{\NN} \,:\,  \sigma \preceq \Psi^{Z\oplus B}  \right\};		
\]
the sets $U_\sigma$ are $\Sigma^0_1(B)$, uniformly. For $m\in \NN$ let
\[
	\sigma_m = \bigcup \left\{ \sigma \,:\,  (\exists x\in I_m)\,\,J^A(x)\downarrow \text{ with use }\sigma\prec A\right\}. 
\]
That is, $\sigma_m$ is an initial segment $\sigma\prec A$ such that for all $x\in I_m$, $J^A(x)\downarrow \LR J^\sigma(x)\downarrow$. Note that since $J^A$ is partial, the map $m\mapsto \sigma_m$ is not $A$-computable, rather, it is $\BLR(A)$: the number of mind-changes is bounded by $|I_m|$. By assumption on~$A$, and the lowness fact mentioned above \cite[Cor.\:2.4]{Kjos.Nies:09}, the map $m\mapsto \sigma_m$ is $\omega$-c.a.; let $(\sigma_{m,s})$ be a computable approximation with a computable bound on the number of mind-changes. (For the current argument, we can allow the approximation $(\sigma_{m,s})$ to be $B$-computable rather than computable, but we need to keep the computable bound on the number of mind-changes.)

We recall some notation: for a $\Sigma^0_1(B)$ set~$V$, and rational $\epsilon >0$, we let $V^{(\le \epsilon)}$ be the result of enumerating~$V$ (with oracle~$B$), up to a point at which we see the measure exceeding~$\epsilon$. Thus, $V\subseteq V^{(\le \epsilon)}$,  $\leb(V^{(\le \epsilon)})\le \epsilon$, and $V = V^{(\le \epsilon)}$ if $\leb V\le \epsilon$. We then let 
\[
	G_m = U_{\sigma_m}^{(\le 2^{-m})},  
\]
and for each~$s$ we let $G_{m,s} = U_{\sigma_{m,s},s}^{(\le 2^{-m})}$, the result of enumerating $U_{\sigma_{m,s}}^{(\le 2^{-m})}$ for~$s$ many steps. So the approximation $(G_{m,s})$ is $B$-computable, and the $m$th component is restarted only when $\sigma_{m,s}\ne \sigma_{m,s+1}$, for which we have a computable bound. Hence, $(G_m)$ is a $\DemBLR \la B\ra $ test. 

We claim that $(G_m)$ is close to what is required in~$(*)$: if $A = \Psi^{B\oplus Y}$ for some~$Y$ that passes the test $(G_m)$, then $J^A$ has an $h$-bounded $B$-c.e.\ trace. The argument is similar to that of \cite[Thm.\:3.2]{Nies:11}. For each~$m$, $x\in I_m$, and~$s$, we enumerate $J^{\sigma_{m,s}}(x)$ into $T_x$ if it is defined at stage~$s$, and the measure of $G_{m,s}$ is exactly $2^{-m}$ (the stage~$s$ version of $G_m$ is ``full''). 

Suppose that $y,z$ are distinct elements of~$T_x$; suppose that we enumerate~$y$ into~$T_x$ at stage~$s$ (so $y=J^{\sigma_{m,s}}(x)$ and $G_{m,s}$ is full at stage~$s$), and we enumerate~$z$ into~$T_x$ at some stage~$t$. 
Since $y\ne z$, the strings $\sigma_{m,s}$ and $\sigma_{m,t}$ are distinct incomparable, which implies that $G_{m,s}$ and $G_{m,t}$ are disjoint. Since both $G_{m,s}$ and $G_{m,t}$ are full, $|T_x| \le 2^m \le h(x)$. Hence, the $B$-c.e.\ trace $(T_x)$ is $h$-bounded. Suppose that $A = \Psi^{Y\oplus B}$ and that $Y\notin G_m$. This means that $\leb(G_m) = 2^{-m}$ (otherwise at some late point we would enumerate~$Y$ into~$G_m$), indeed, for all but finitely many stages~$s$, $\leb G_{m,s} = 2^{-m}$. Hence, for all $x\in I_m$, if $J^A(x)$ is defined then $J^A(x)\in T_x$. So if $A=\Psi^{B\oplus Y}$ for some~$Y$ that passes $(G_m)$, then a finite modification of $(T_x)$ traces $J^A$. 

Toward~$(*)$, we use the fact that since $B$ is computable from a c.e., jump-traceable set, it is \emph{low for $\DemBLR$} (\cite[Thm.\:1.8, Prop.\:4.3]{Bienvenu.Downey.ea:14}): for every $\DemBLR \la B \ra$ test $(G_m)$, there is an unrelativised Demuth test $(H_m)$ that \emph{covers} $(G_m)$ in the sense that every set passing $(H_m)$ also passes $(G_m)$. Hence, there is a Demuth test $(H_m)$ with the same property as $(G_m)$: if $A = \Psi^{B\oplus Y}$ for some~$Y$ that passes $(H_m)$, then $J^A$ has an $h$-bounded $B$-c.e.\ trace. 

Finally, the argument of \cite[Lem.\:2.6]{Nies:11} gives~$(*)$. Let $(\Phi_e)$ be an effective listing of all Turing functionals, and let $\Psi^{0^e1X} = \Phi_e^X$ for all~$e$ and~$X$. Let $(H_m)$ be a Demuth test obtained for this functional~$\Psi$. The referenced lemma states that there is a Demuth test $(S_m)$ such that for all~$Y$, if $Y$ passes $(S_m)$ then for all~$e$, $0^e1Y$ passes $(H_m)$. Then $(S_m)$ is as required for~$(*)$; this completes the proof of (b)$_{\+C}\RA$(a) for all Demuth-compatible classes~$\+ C$.

\subsection{Proof of (b)$\RA$(a) of \cref{thm:random_computing_characterisation}, case (5)}

We now consider the remaining case of \cref{thm:random_computing_characterisation}: if $A\le_\Tur B\oplus \Omega_R$ for every infinite computable set~$R$,  then $A\SJR B$. In fact, using this assumption we prove that $A\models_B \cost$ for every benign cost function~$\cost$, and use \cref{thm:cost_function_characterisation}, which applies since every $K$-trivial set is jump-traceable. We generally follow arguments in \cite{Greenberg.Miller.ea:24}, which gives the result when $B = \emptyset$ . 

Let $R$ be infinite and computable. For $s>n$ we let 
\[
	k_s(n) = \floor{ - \log_2 (\Omega_s -\Omega_n)}, 
\]
and 
\[
	\cost_{\Omega,R}(n,s) = 2^{-|R\cap k_s(n)|}
\]
(see \cite[Def.\:6.3]{Greenberg.Miller.ea:24}). The cost function $\cost_{\Omega,R}$ is benign (\cite[Prop.\:9.2]{Greenberg.Miller.ea:24}). Further, if~$\cost$ is any benign cost function, then there is some infinite computable~$R$ such that for all~$n$, $\lim_s \cost(n,s)\le^\times \lim_s\cost_{\Omega,R}(n,s)$ (\cite[Prop.\:9.3]{Greenberg.Miller.ea:24}). By \cite[Thm.\:3.4]{Nies:17}, this implies that if $A\models_B \cost_{\Omega,R}$ then $A\models_B \cost$. Hence, to prove the desired implication, the following suffices:

\begin{lemma} \label{lem:6.9}
	Let~$A$ and~$B$ be $K$-trivial, and $R$ be infinite and computable. If $A\le_\Tur B\oplus \Omega_R$ then $A\models_B \cost_{\Omega,R}$. 
\end{lemma}

This is the required partial relativisation of \cite[Lem.\:6.9]{Greenberg.Miller.ea:24}. We will follow the proof of this lemma. As it is fairly long, we will not copy all the details; rather, we will give the overall structure of the proof, and indicate how it needs to be modified to accommodate $B\ne \emptyset$. 

The proof has two main steps. We first prove \cref{lem:6.9} assuming that~$A$ is c.e.\ relative to~$B$; we then show how to remove this extra assumption.

\subsubsection*{Proof of \cref{lem:6.9} assuming that $A$ is c.e.\ in~$B$}
We may assume that~$R$ is co-infinite; otherwise, $\cost_{\Omega,R}=^\times \cost_{\Omega}$, and since~$A$ is $K$-trivial, it obeys $\cost_\Omega$ (with no need for $B$'s help). For $n\in \NN$ let
\[
	U_n  = \bigcup_{s\ge n}  \big[ (\Omega_s)_{R^\complement}\rest{|R^\complement \cap n|}  \big]
\]
(the string in the definition is the result of erasing the bits in locations in~$R$ from the string $\Omega_s\rest{n}$); let $U_{n,s}$ be the stage~$s$ enumeration of~$U_n$. Then $\leb(U_{n,s}) \le^\times \cost_{\Omega,R^\complement}(n,s)$ and $\Omega_{R^\complement}\in \bigcap_n U_n$. 
Following the notation in \cite{Greenberg.Miller.ea:24}, let $Y = \Omega_R$ and $X = \Omega_{R^\complement}$. Fix a $B$-computable functional~$\Phi$ such that $A = \Phi^Y$. Fixing a $B$-computable enumeration $(A_s)$ of~$A$, we let $E_s$, the stage~$s$ error set, be the collection of~$Z$ such that $\Phi_s^Z$ lies strictly to the left of $A_s$, and $E = \bigcup_s E_s$ is the set of~$Z$ such that $\Phi^Z$ lies to the left of~$A$. We let $Q_s = 2^\omega\times (2^\omega\setminus E_s)$ (and $Q = 2^\omega\times (2^\omega\setminus E)$). We perform the ``ravenous sets'' construction precisely as in \cite{Greenberg.Miller.ea:24}, except that now the construction is computable in~$B$. We obtain sets $V^k_n$, and let $V^k = \bigcup_n V^k_n$. 

The main part of the verification that we need to discuss is that $(X,Y)\notin \bigcap_k V^k\cap Q$. The first step of this is the equivalence between that statement, and $Q$ having positive density at $(X,Y)$. When $B = \emptyset$ this is \cite[Lem.\:3.3]{Bienvenu.Hoelzl.ea:14} (quoted as \cite[Fact\:8.1]{Greenberg.Miller.ea:24}). The relativisation of this fact gives:
\begin{description}
	\item[$(\otimes)_1$] Suppose that $Y$ is random relative to~$B$, $P$ is a $\Pi^0_1(B)$-class, and $Y\in P$. Then $P$ has density~0 at~$Y$ if and only if $Y$ fails a $B$-difference test on~$P$. 
\end{description}
Here a $B$-difference test on~$P$ is a test of the form $(P\cap O_n)$, where $(O_n)$ are uniformly $\Sigma^0_1(B)$ and $\leb(P\cap O_n)\le 2^{-n}$. We also need the following relativisation of a characterisation of difference randomness from \cite{Franklin.Ng:10}: if $Y$ is a  ML-random set,
\begin{description}
	\item[$(\otimes)_2$]   $Y$ passes all $B$-difference tests if and only if $Y\oplus B\ge_\Tur B'$. 
\end{description}
Now we argue as follows. Since $Y$ is ML-random and $B$ is $K$-trivial, $Y$ is $B$-random. Since~$B$ is $K$-trivial, it is low, i.e., $B'\equiv_\Tur \emptyset'$. Now~$Y$ is Turing incomplete (since~$R$ is co-infinite), that is, $Y\nge_\Tur \emptyset'$. By a result of Day and Miller \cite{Day.Miller:14}, since~$B$ is $K$-trivial, $Y\oplus B\nge_\Tur \emptyset'$ (this result is not needed in \cite{Greenberg.Miller.ea:24}). By $(\otimes)_2$, $Y$ passes all $B$-difference tests. Since $Y\notin E$, by $(\otimes)_1$, $2^\omega\setminus E$  has positive density at~$Y$ (again we use the fact that $Y$ is $B$-random, as $B$ is $K$-trivial). It follows that $Q$ has positive density at $(X,Y)$. Again since $(X,Y)$ is ML-random, it is $B$-random. By $(\otimes)_1$ again, $(X,Y)\notin \bigcap_k V^k\cap Q$, as required. 

The rest of the verification in \cite{Greenberg.Miller.ea:24} goes through when relativising to~$B$, finishing the proof of \cref{lem:6.9} when $A$ is c.e.\ relative to~$B$.

\subsubsection*{Proof of \cref{lem:6.9} without extra assumptions}
As in \cite{Greenberg.Miller.ea:24}, this case follows from the special case above using two facts.
\begin{description}
	\item[$(\otimes)_3$] If $A\le_\Tur B\oplus \Omega_R$ then there is some $B$-c.e., $K$-trivial set~$D$ such that $A\le_\Tur B\oplus D$ and $D\le_\Tur B\oplus \Omega_R$. 

	\item[$(\otimes)_4$] If~$D$ is $K$-trivial, $D\models_B \cost_{\Omega,R}$ and $A\le_\Tur D\oplus B$ then $A\models_B \cost_{\Omega,R}$. 
\end{description}

We start with the latter, following the proof in \cite[Lem.\:8.2]{Greenberg.Miller.ea:24} (which is similar to \cite[Prop.\:2.3]{Greenberg.Miller.etal:19}). Let~$\psi$ be the use of a computation of~$A$ from $D\oplus B$. The $K$-trivial degrees are closed under taking joins, so $D\oplus B$ is $K$-trivial. Hence \cite[Lem.\:2.5]{Barmpalias.Downey:14} applies, so $\Omega-\Omega_n \le^\times \Omega-\Omega_{\psi(n)}$. The rest of the proof of \cite[Lem.\:8.2]{Greenberg.Miller.ea:24} follows without changes. 

We turn to verify $(\otimes)_3$, which follows from a partial relativisation of \cite[Thm.\:3.1]{Greenberg.Miller.ea:24}. For this, we need a partial relativisation of \cite[Lem.\:3.2]{Greenberg.Miller.ea:24}:

\begin{description}
	\item[$(\otimes)_5$] Suppose that $A$ and~$B$ are $K$-trivial. There is a $B$-computable approximation $(A_s)$ of~$A$ such that if $X$ is ML-random and $A = \Phi(X,B)$ then for almost all~$n$, if $A\rest{n}\preceq \Phi_s(X,B)$ then $A\rest{n} = A_t\rest{n}$ for all $t\ge s$. 
\end{description}
This follows from the proof in \cite{Greenberg.Miller.ea:24}, using a partially relativised version of the ``main lemma'' derived from the golden run construction \cite[5.5.1]{Nies:book}; we need the prefix-free machine~$M$ to be $B$-computable, and the sequence of stages $(q(i))$ to be $B$-computable. Again, we follow the golden run construction without changes, relativising all to~$B$. We use the fact that $B$ is low for~$K$, so $A$ is $K$-trivial relative to~$B$: $K^B(A\rest{n}) =^+ K(A\rest{n}) =^+ K(n) =^+ K^B(n)$. The added constants are incorporated into the construction. 

This completes the proof of \cref{lem:6.9}, and so of \cref{thm:random_computing_characterisation}.

\section{Concluding remarks and open questions} \label{sec:conclu}

In this section we mention some corollaries of our main theorems, discuss related results, and state some open questions.  

The very first characterisation  of strong jump-traceability was  in terms of Kolmogorov complexity~\cite{Figueira.ea:08}. This extends to a characterisation of  the weak reducibility $\SJR$. Informally,   $A\SJR B$ if and only if  the plain descriptive string complexity $C^A$ is ``almost'' lower-bounded by~$C^B$. The formal version follows.

\begin{proposition} The following are equivalent for all sets $A,B$. 
	\begin{itemize}
		\item[(a)] $A \SJR B$;
		\item[(b)] $\fa x \ [C^B(x) \lep  C^A(x) +  h(C^A(x))]$ for each order function $h$.  
	\end{itemize}
\end{proposition}

\cite[Thm.\ 17]{Figueira.ea:08} states this for the case that $B= \ES$ (also see  \cite[Cor 8.4.32]{Nies:book}). The proof gives the required partial relativisation, where we still almost lower-bound $C^A$ (rather than $C^{A \oplus B}$), and the functions $h$ are computable.

\subsection{The structure of the SJT degrees}

Not much is known about the degree structure given by $\SJR$. We can state the following: 

\begin{proposition} \label{prop:countable}
If $B$ is jump traceable, then the set $\{ A : A \SJR B\}$ is countable.
\end{proposition}

In particular, the $\SJR$-degree of~$B$ is countable, and it bounds only countably many $\SJR$-degrees. 

\begin{proof}
By \cref{thm:cost_function_characterisation}, if $B$ is jump traceable and $A \SJR B$, then $A \models_B \cost$ for any benign cost function.  In particular, $A \in \Delta^0_2(B)$, which is a countable set.
\end{proof}

In contrast, not all initial segments of the degree structure induced by  $\SJR$ are countable:
Ng~\cite[Th.\ 6.2.1]{Ng:thesis} proved that 
the set $\{A : A \SJR \emptyset'\}$ contains a perfect class.
 

Like the $K$-trivials, the strongly jump-traceable degrees form an ideal. In particular, they are closed under taking joins. We ask:

\begin{question} \label{question:join}
	Suppose that~$B$ is jump-traceable. Is  the class $\{ A : A \SJR B\}$ closed under taking joins?
\end{question}

By the argument of the proof of \cref{thm:A_to_A_join_B} (see Section~\ref{ss:proof Thm 1.6}), for an affirmative answer to \cref{question:join}, it would suffice to show that if $A_0,A_1\SJR B$ and $B$ is jump-traceable then $A_0\oplus A_1 \le_{\textup{JT}} B$. We could even hope for the following: there is a benign cost function~$\cost$ such that if $B$ is jump-traceable and $A\models_B \cost$ then $A$ is jump-traceable (equivalently $A\le_{\textup{JT}} B$). We remark that \cref{question:join} for $B=\emptyset$ was first proved directly, without using cost functions (see \cite{Cholak.Downey.ea:08} for the c.e.\ case); it would be interesting to see if ideas from that argument would be useful. 

\smallskip

It is open whether the c.e.\ SJT-degrees are dense. We conjecture that  for each  c.e.\ set~$E$ that is not strongly jump traceable,  there are c.e.\ sets $A,B \le_\Tur E$ that are $\SJR$-incomparable.  A certain obstacle to showing such structural results  is   the theorem of Ng~\cite{Ng:thesis} that the least degree in the c.e.\ SJT-degrees, given by the  strongly-jump traceable c.e.\ sets, has  a $\PI 4$ complete index set. Examining the definition, one sees that the arithmetical complexity of the reducibility itself on the c.e.\ sets is also $\PI 4$.   In contrast, Turing reducibility is $\SI 4$, and the least degree is merely $\SI 3$.

\subsection{Relationship with LR-reducibility}

Recall the weak reducibility $\le_\LRR$ mentioned in the introduction: $A\le_{\LRR} B$ if $\MLR^B\subseteq \MLR^A$. Every strongly jump-traceable set is $K$-trivial, that is, if $A\SJR \emptyset$ then $A \le_\LRR \emptyset$. 

\begin{question}
	Does $ A \SJR B$ imply $A\le_{\LRR} B$?
\end{question}

Here \cref{thm:A_to_A_join_B} may be relevant: the relation  $A\oplus B \le_{\LRR} B$ is much better understood than $A\le_{\LRR} B$, as it has a characterisation in terms of $K$-triviality. In particular, this characterisation implies that every $\le_{\LRR}$-degree is countable, even though some $\le_{\LRR}$-degrees bound uncountably many such degrees. In contrast, we do not know whether every $\SJR$-degree is countable.

In general, it would be interesting to give characterisations of other weak reducibilities using the three paradigms discussed in the introduction. In particular, we suggest the reducibilities $A\le_\LRR B$, $A\le_{\textup{JT}} B$, and $A'\le_{\textup{tt}} B'$; see \cite[Sec.\:8.4]{Nies:book}.

\subsection{Relationship with ML-reducibility}

The following relation is a main topic of the article \cite{Greenberg.Miller.ea:24}:

\begin{definition} \label{def:ML_reducibility}
	For sets $A,B \subseteq \NN$, one  writes $A\le_{\textup{ML}} B$ if for every ML-random~$Y$, $Y\ge_\Tur B$ implies $Y\ge_\Tur A$. 
\end{definition}

This relation is particularly useful in understanding the   structure of the $K$-trivial Turing degrees. It satisfies the conditions for being a weak reducibility, except that it is not known to be  an arithmetical relation, even when restricted to the $K$-trivials. 

There are certain differences between the reducibilities $\le_{\textup{ML}}$ and $\SJR$ on the $K$-trivial sets. For example, the least $\le_{\textup{ML}}$-degree consists only of the computable sets, not all strongly jump-traceable sets (this shows that $\SJR$ does not imply $\le_{\textup{ML}}$ on the $K$-trivials). Further, there is a greatest $\le_{\textup{ML}}$-degree of $K$-trivials (called the ``smart'' $K$-trivials in \cite{Bienvenu.Greenberg.ea:16,Greenberg.Miller.ea:24}).  By the following proposition there is no greatest $\SJR$-degree among the $K$-trivials; in particular, $\le_{\textup{ML}}$ does not imply $\SJR$ on the $K$-trivials.

\begin{proposition} \label{prop:K-triv} For each $K$-trivial set $B$, there is a c.e.\ $K$-trivial set  $C \ge_T B$ such that $C \not \SJR B$. \end{proposition}
\begin{proof}  By   \cite{Cholak.Downey.ea:08},  for some fixed computable function $h$ there is a functional  $\Psi$ and $K$-trivial set $A$ such that $\Psi^A$ has no c.e.\ trace bounded by $h$. (Also see \cite[8.5.1]{Nies:book} where this is shown for $h(n)= 0.5 \log \log n$.) Since there is a computable function $p$ such that  $\Psi^X(n)= J^X(p(n))$ for each $X,n$,   we get a    computable function $g$ so that the statement holds for $J$ and $g$ instead of $\Psi$ and $h$. Relativizing to $B$ we can retain the same $g$, so    for each $B$ there is a  set $A$ that is $K$-trivial in $B$, such that $J^{\hat C}$ does not have a $B$-c.e.\ trace bounded by $g$, where $\hat C=A \oplus B$. In particular, $\hat C \not \SJR B$. 

If $B$ is $K$-trivial, it is low for $K$, and hence $A$ is also $K$-trivial. Therefore $\hat C$ is $K$-trivial.  To conclude the proof, take a  $K$-trivial c.e.\ set $ C \ge_T \hat C$.  (For standard facts on $K$-trivial sets see, for instance, \cite[Ch.\ 5]{Nies:book}.)
\end{proof}

We conjecture that for each benign  cost function $\cost$, for each    $B\models \cost$,  there is a c.e.\   set   $A \models \cost$ such that $A \not \SJR B$. This would strengthen \cref{prop:K-triv}.  

\smallskip

By \cite[Thm.\:3.1]{Greenberg.Miller.ea:24}, every ML-degree of a $K$-trivial set contains a c.e.\ set. By \cite{Diamondstone.Greenberg.ea:15}, every strongly jump-traceable set is computable from a c.e.\ strongly jump-traceable set. It is natural to ask:

\begin{question}
	Does every $K$-trivial $\SJR$-degree contain a c.e.\ set?
\end{question}

A reducibility strictly weaker than $\le_{\textup{ML}}$ is $A\le_{\omega-\text{ML}} B$: every $\omega$-c.a.\ ML-random Turing above $B$ is also above $A$. This reducibility is somewhat closer to $\SJR$. \Cref{thm:random_computing_characterisation} for the class $\+C$ being the $\omega$-c.a.\ sets implies:

\begin{proposition} \label{prop:SJT_and_omega_ML}
 	Let $A$ and~$B$ be $K$-trivial.
 	\begin{enumerate}
 		\item[(1)] If $A\SJR B$  then $A\le_{\omega-\textup{ML}} B$. 
 		\item[(2)] $A\SJR \emptyset$ if and only if $A\le_{\omega-\textup{ML}} \emptyset$. 
 	\end{enumerate}
 \end{proposition}

Finally, we mention a related notion. For $\+C \sub \+ P(\NN) $  one defines $\+ C^\Diamond$ as the class of c.e.\ sets that are Turing below each ML-random set in $\+ C$~\cite[Section 8.5]{Nies:book}. \Cref{thm:random_computing_characterisation} implies:

\begin{corollary} Let $\+C $ be a nonempty class of ML-randoms that contains no weakly Demuth random. Then $\+ C^\Diamond$ is downward closed under $\SJR$. \end{corollary}
For instance, let $\+ C  = \{\Omega_R\}$ for a co-infinite computable set $R$. This shows that the subideals, in the Turing sense,  of the $K$-trivials considered in \cite{Greenberg.Miller.etal:19,Greenberg.Miller.ea:24} are actually closed  downward under the weaker SJT-reducibility.

\bibliographystyle{plain}

\end{document}